\documentclass[12pt,a4paper,oneside]{article}
\usepackage[centertags]{amsmath}
\usepackage{amssymb,mathtools,mathrsfs}
\usepackage{amsthm}
\usepackage{enumerate}
\usepackage{stmaryrd}
\usepackage{dsfont}
\usepackage{url}
\usepackage{graphicx}
\usepackage{color}
\usepackage{epsfig}
\usepackage{subcaption}
\usepackage[a4paper,tmargin=3cm,bmargin=3cm,rmargin=2.2cm,lmargin=2.2cm]{geometry}

\definecolor{dg}{rgb}{0,0.5,0}		
\definecolor{orange}{rgb}{1.0, 0.5, 0.0}

\usepackage{hyphenat}
\usepackage[english]{babel}
\usepackage[T1]{fontenc}

\theoremstyle{theorem}
\newtheorem{Thm}{Theorem}
\newtheorem{Cor}[Thm]{Corollary}

\newtheorem{Prop}[Thm]{Proposition}

\theoremstyle{definition}

\newtheorem{Rem}[Thm]{Remark}

\newcommand{\sZ}{{\mathbb Z}}
\newcommand{\sR}{{\mathbb R}}
\newcommand{\sC}{{\mathbb C}}

\renewcommand{\L}{\mathcal{L}}

\newcommand{\itSigma}{\mathit{\Sigma}}

\newcommand{\ii}{\textup{i}}
\newcommand{\ee}{\textup{e}}

\DeclareMathOperator{\diag}{diag}
\DeclareMathOperator{\tr}{tr}

\DeclareMathOperator{\sgn}{sgn}

\DeclareMathOperator{\Arg}{Arg}
\DeclareMathOperator{\RE}{Re}
\DeclareMathOperator{\IM}{Im}
\DeclareMathOperator{\Pid}{\Pi_{\textup{d}}}

\newcommand{\vc}{\vcentcolon =}
\newcommand{\cv}{= \vcentcolon}

\renewcommand{\bar}{\overline}		

\begin{document}

\title{On the diagonal product of special unitary matrices}
\author{Tomasz Miller\thanks{Email: tomasz.miller@uj.edu.pl}\\ \normalsize Copernicus Center for Interdisciplinary Studies, Jagiellonian University \\ 
\normalsize Szczepa\'{n}ska 1/5, 31-011 Cracow, Poland }

\date{\today}

\maketitle

\begin{abstract}
We study the image of $\textup{SU}(n)$ under the diagonal product map. Using only elementary tools of linear algebra, analysis and general topology, we find the analytical formula for the boundary of this image and find all special unitary matrices whose diagonal product lies on that boundary.
\end{abstract}




\section{Introduction and main result}

The \emph{diagonal product} of an $n$-by-$n$ matrix $A$ is defined simply as the product of all its diagonal entries\footnote{Although some authors \cite{sinkhorn69,engel73} understand this term broadlier as $\Pi_{\sigma}(A) = \prod\nolimits_j A_{j\sigma_j} = \Pid(P_\sigma A)$ for any given permutation $\sigma \in S_n$, where $P_\sigma$ is its corresponding permutation matrix.} $\Pid(A) := \prod\nolimits_{j=1}^n A_{jj}$. Although not as ubiquitous as its `additive cousin' -- the trace -- the diagonal product does appear throughtout matrix theory, most notably in the classic Hadamard's inequality $\det H \leq \Pid(H)$ valid for any positive-semidefinite $H$ (as well as in some of its generalizations \cite{rozanski17}), or even in the very definitions of the determinant $\det A = \sum\nolimits_{\sigma \in S_n} \sgn\sigma \Pid(P_\sigma A)$, the permanent and other immanants. The diagonal product appears also in some natural generalizations of the notion of numerical range (cf. \cite{miranda92, bebiano93} and references therein).

In \cite{kaiser06} it was shown that $\tr(\textup{SU}(n))$, i.e. the image of $\textup{SU}(n)$ under the trace map, is quite nontrivial for $n \geq 3$, being a compact region enclosed by the $n$-cusped hypocycloid. In this note we ask similar questions for the diagonal product map: What is the image $\Pid(\textup{SU}(n))$? Can its boundary be described analitycally? What are the matrices whose diagonal product lies on that boundary? Let us observe that, unlike in the trace case, these questions do not translate directly to relatively simple algebraical relations between the matrices' eigenvalues.

Before turning to special unitary matrices, let us observe that the image of the set $\textup{U}(n)$ of all unitary $n$-by-$n$ matrices under the map $\Pid$ is not very interesting. For $n=1$ it is trivially the unit circle, whereas for $n > 1$ we have the following result.
\begin{Prop}
For any fixed $n \geq 2$, the image $\Pid(\textup{U}(n))$ is the closed unit disk. Moreover, the diagonal product of $U\in\textup{U}(n)$ lies on the unit circle iff $U$ is diagonal. In other words,
\begin{align}
\label{premain}
|\Pid(U)| = 1 \quad \Longleftrightarrow \quad U = \diag(\ee^{\ii \alpha_1}, \ldots, \ee^{\ii \alpha_n}).
\end{align} 
\end{Prop}
\begin{proof}
For any $U \in \textup{U}(n)$, by definition one has that $\sum_k |U_{jk}|^2 = 1$ for any $j$, and therefore $|U_{jj}| \leq 1$. This allows us to write $|\Pid(U)| = \prod\nolimits_j |U_{jj}| \leq 1$. Of course, every diagonal $U$ saturates the latter bound.

Suppose now that $|\Pid(U)| = 1$, but $U \in \textup{U}(n)$ is \emph{not} diagonal, i.e., $U_{jk} \neq 0$ for some $j \neq k$. But since $\sum_j |U_{jk}|^2 = 1 = \sum_k |U_{jk}|^2$, this would imply that $|U_{jj}|, |U_{kk}| < 1$ and we would get $|\Pid(U)| = \prod\nolimits_l |U_{ll}| < 1$, contradicting the assumption on $U$.

Finally, in order to show that every $z$ from the closed unit disk is a diagonal product of some $U \in \textup{U}(n)$, define the following $2$-by-$2$ unitary matrix
\begin{align*}
U_z \vc \left[ \begin{array}{cc} \sqrt{|z|}  & -\sqrt{1 - |z|} \\ \sqrt{1 - |z|} & \sqrt{|z|} \end{array} \right] \diag(\sgn z, 1),
\end{align*}
where $\sgn z$ is the complex signum function\footnote{It is defined as $\sgn z \vc z/|z|$ for $z \neq 0$ and $\sgn 0 \vc 0$.}. Then, immediately, $\Pid(U_z \oplus I_{n-2}) = z$.
\end{proof}

For $\textup{SU}(n)$ the situation gets much more interesting. Concretely, the main result of this note is the following.
\begin{Thm}
\label{main}
For any fixed $n \geq 1$, the image $\Pid(\textup{SU}(n))$ is the compact region of the complex plane enclosed by the (image of the) closed curve
\begin{align}
\label{main1}
\gamma: [-\pi,\pi] \rightarrow \sC, \qquad \gamma(\alpha) = \ee^{\ii \alpha} \left( 1 - \left( 1 - \ee^{-\ii \alpha} \right) \tfrac{1}{n} \right)^n.
\end{align}
Moreover, the diagonal product of $U \in \textup{SU}(n)$ is equal to $\gamma(\alpha)$ for some $\alpha \in [-\pi,\pi]$ iff $U$ is of the form
\begin{align}
\label{main2}
U = \left( I - (1 - \ee^{-\ii \alpha}) v v^\dag \right) \diag(\ee^{\ii \alpha_1}, \ldots, \ee^{\ii \alpha_n}),
\end{align}
with $v \in \sC^n$ satisfying $|v_k| = \tfrac{1}{\sqrt{n}}$ for all $k$, and $\ee^{\ii \sum_k \alpha_k} = \ee^{\ii \alpha}$.
\end{Thm}

\begin{figure}[h]
\centering
  \includegraphics[width=.95\linewidth]{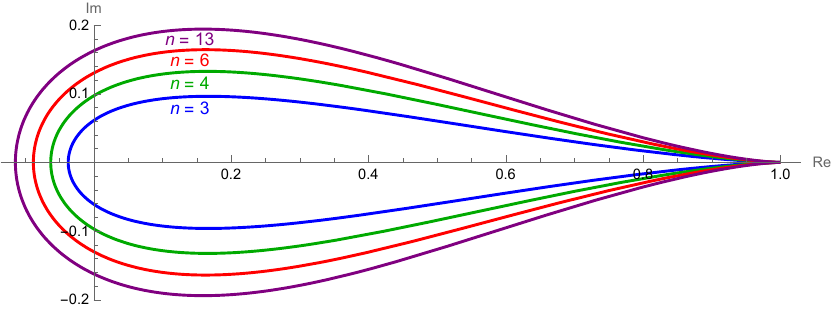}
\caption{\label{fig4}}Graphs of $\gamma$ for several values of $n$.
\end{figure}

\begin{Rem}
\label{rem2}
For $n > 1$ the curve $\gamma$ is regular everywhere except at $\gamma(0) = 1$. Indeed, one has
\begin{align*}
\partial_\alpha \gamma = \ii \left(1-\tfrac{1}{n}\right) \left( 1 - \ee^{-\ii \alpha} \right) \ee^{\ii \alpha} \left( 1 - \left( 1 - \ee^{-\ii \alpha} \right) \tfrac{1}{n} \right)^{n-1},
\end{align*}
and so $|\partial_\alpha \gamma| = 0$ iff $\alpha = 0$.

Let us emphasize that the parameter $\alpha$ in formula (\ref{main1}) is \emph{not} the polar angle. Notably however, for $n > 2$ the curve $\gamma$ can be parametrized in the polar coordinates $r, \theta$ via
\begin{align}
\label{polar1}
r(\theta) = \left( 1 - \tfrac{4(n-1)}{n^2} \sin^2 \tfrac{\alpha(\theta)}{2} \right)^{n/2},
\end{align}
where $\alpha(\theta)$ denotes the inverse function of the smooth increasing bijection of $[-\pi,\pi]$ onto itself defined by 
\begin{align}
\label{polar2}
\theta(\alpha) = \alpha - n \arctan\tfrac{\sin\alpha}{n-1+\cos\alpha}.
\end{align}
To see that $\theta(\alpha)$ is indeed increasing, notice that its derivative can be written as $\theta'(\alpha) = \frac{2(n-1)(n-2)\sin^2\tfrac{\alpha}{2}}{(n-2)^2 + 4(n-1)\cos^2\tfrac{\alpha}{2}}$, what is positive for all $\alpha$ except $0$. This implies that $\alpha(\theta)$ is itself a continuous\footnote{It is also smooth everywhere except at $\theta = 0$, where $\alpha'$ does not exist.} increasing bijection. Notice also that $\theta(\alpha) = \alpha$ only at $0$ and $\pm \pi$.
\end{Rem}

\section{Proof of the main result}

In the proof of Theorem \ref{main} we shall use only elementary tools of linear algebra, analysis and general topology. For the sake of legibility it is divided into eight steps. The first step concerns the cases $n=1$ and $n=2$, which have to be analyzed directly, being in some sense degenerate.

\subsubsection*{Step 1: The case $n < 3$.}

In the trivial $n=1$ case formula (\ref{main1}) yields the constant curve $\gamma = 1$, whereas formula (\ref{main2}) reduces to $U = 1$ for any $\alpha$, as they should.

For $n=2$, every $U \in \textup{SU}(2)$ can be written as
\begin{align}
\label{2case}
U = \left[ \begin{array}{cc} z & -\overline{w} \\ w & \overline{z} \end{array} \right] \quad \textrm{with} \quad |z|^2 + |w|^2 = 1,
\end{align}
and so clearly $\Pid(\textup{SU}(2))$ is the real interval $[0,1]$. Formula (\ref{main1}) agrees with that, yielding $\gamma(\alpha) = \cos^2\tfrac{\alpha}{2}$. In order to demonstrate the second part of the theorem, we have to show how to reexpress (\ref{2case}) in the form given by (\ref{main2}) for any complex $z,w$ satisfying $|z|^2 + |w|^2 = 1$. As one can easily check, this can be done, e.g., as follows:
\begin{align*}
& \textnormal{If } w = 0 \textnormal{, then take} \ \ \ee^{\ii \alpha_1} \vc z, \ \ \ee^{\ii \alpha_2} \vc \bar{z}, \ \ \textnormal{and \emph{any} } v \in \sC^2.
\\
& \textnormal{If } z = 0 \textnormal{, then take} \ \ \ee^{\ii \alpha_1} \vc -1, \ \ \ee^{\ii \alpha_2} \vc 1 \ \ \textnormal{and } v \vc \tfrac{1}{\sqrt{2}} \left[\begin{smallmatrix} 1 \\ w \end{smallmatrix} \right].
\\
& \textnormal{If } z,w \neq 0 \textnormal{, take} \ \ \ee^{\ii \alpha_1} \vc (|z| + \ii |w|)\sgn z, \ \ \ee^{\ii \alpha_2} \vc (|z| + \ii |w|)\sgn \bar{z} \ \ \textnormal{and } v \vc \tfrac{1}{\sqrt{2}} \left[\begin{smallmatrix} \sgn z \\ \ii \sgn w \end{smallmatrix} \right].
\end{align*}

\subsubsection*{Step 2: The constrained maximization problem.}

From now on we assume that $n \geq 3$. Our strategy is to fix the polar angle $\theta \in [-\pi,\pi]$ and see how far stretches the $\theta$-rotated positive real axis intersected with $\Pid(\textup{SU}(n))$. More formally, we shall consider the following constrained maximization problem:
\begin{align}
\label{max1}
\textnormal{maximize} \quad & F(U) \vc \RE(\ee^{-\ii \theta} \Pid(U)) \quad \textnormal{over } \textup{SU}(n)
\\
\nonumber
\textnormal{subject to} \quad & G(U) \vc \IM(\ee^{-\ii \theta} \Pid(U)) = 0.
\end{align}
The constrained maximum of course exists (by the extreme value theorem). However, what we want to prove (among other things) is that the constrained maximum equals $|\gamma(\alpha(\theta))|$ with $\alpha(\theta)$ defined in Remark \ref{rem2}, and is attained only at $U$ of the form given by (\ref{main2}).

Let us begin by considering the following special unitary matrix
\begin{align*}
U_\theta \vc \ee^{\ii \alpha(\theta)/n} \left( I - (1 - \ee^{-\ii \alpha(\theta)}) v v^T \right), \quad \textnormal{where } v_k = \tfrac{1}{\sqrt{n}}, \ \forall k.
\end{align*}
By Remark \ref{rem2}, its diagonal product equals $\Pid(U_\theta) = \ee^{\ii\theta}|\gamma(\alpha(\theta))|$, and so it satisfies the constraint $G(U_\theta) = 0$ and yields $F(U_\theta) = |\gamma(\alpha(\theta))|$. Therefore, the constrained maximum for problem (\ref{max1}) is \emph{positive}.

Let now $U_0 \in \textup{SU}(n)$ be a maximizer solving problem (\ref{max1}). Since for any $n$-by-$n$ matrix $A$ we have that
\begin{align}
\label{obs}
\Pid(A) = \Pid(A\diag(\ee^{\ii \beta_1},\ldots,\ee^{\ii \beta_n})) \quad \textnormal{provided} \quad \ee^{\ii \sum_k \beta_k} = 1,
\end{align}
we can assume without loss of generality that the diagonal entries of $U_0$ have all the same argument which, by the fact that $F(U_0) > 0$ and $G(U_0) = 0$, must be equal to $\theta/n$. This allows us to define a unitary matrix
\begin{align}
\label{prod4}
U_1 \vc \ee^{-\ii \theta/n} U_0
\end{align}
with determinant $\ee^{-\ii \theta}$ and the property that all its diagonal entries are real positive numbers. In other words, denoting $U_1 = [c_{jk}]_{j,k=1,\ldots,n}$, we have
\begin{align}
\label{prod5}
c_{jj} > 0, \quad \forall j 
\end{align}
Our aim now is to extract some information about $c_{jk}$'s from the fact that $U_0$ is a maximizer solving problem (\ref{max1}).

\subsubsection*{Step 3: Proving that all $c_{jj}$'s are equal.}

We claim that all diagonal entries of $U_1$ are in fact equal. We shall argue by supposing the contrary. Concretely, to fix attention, let us suppose that $c_{11} \neq c_{22}$ and consider the following problem:
\begin{align}
\label{max2}
\textnormal{maximize} \quad & f(s,t) \vc \RE(\ee^{-\ii \theta} \Pid ((\ee^{sX + tY} \oplus I_{n-2}) U_0 (\ee^{-sX - tY} \oplus I_{n-2})) )
\\
\nonumber
\textnormal{subject to} \quad & g(s,t) \vc \IM(\ee^{-\ii \theta} \Pid ((\ee^{sX + tY} \oplus I_{n-2}) U_0 (\ee^{-sX - tY} \oplus I_{n-2})) ) = 0,
\end{align}
where $X \vc \left[\begin{smallmatrix} 0 & -1 \\ 1 & 0 \end{smallmatrix} \right]$ and $Y \vc \left[\begin{smallmatrix} 0 & \ii \\ \ii & 0 \end{smallmatrix} \right]$. It is in fact a subproblem of (\ref{max1}), in which the maximizer $U_0$ is perturbed in a specific way.

It is convenient to introduce $M \vc \left[\begin{smallmatrix} c_{11} & c_{12} \\ c_{21} & c_{22} \end{smallmatrix} \right]$ and, using (\ref{prod4}) and (\ref{prod5}), express the functions $f, g$ in (\ref{max2}) as
\begin{align*}
& f(s,t) = C \RE\left[\left(\ee^{sX+tY} M \ee^{-sX-tY}\right)_{11} \left(\ee^{sX+tY} M \ee^{-sX-tY}\right)_{22} \right],
\\
& g(s,t) = C \IM\left[\left(\ee^{sX+tY} M \ee^{-sX-tY}\right)_{11} \left(\ee^{sX+tY} M \ee^{-sX-tY}\right)_{22} \right],
\end{align*}
where $C \vc \prod\nolimits_{j=3}^n c_{jj} > 0$.

By the very definition of $U_0$, one solution to problem (\ref{max2}) is attained at the point $(s_0,t_0) = (0,0)$. We shall utilize this fact with the help of the Lagrange multipliers method. Concretely, let us construct the Lagrangian function
\begin{align}
\label{prod9}
\L(s,t;\lambda) & \vc f(s,t) + \lambda g(s,t) 
\\
\nonumber
& = C \RE\left[(1 - \lambda\ii) \left(\ee^{sX+tY} M \ee^{-sX-tY}\right)_{11} \left(\ee^{sX+tY} M \ee^{-sX-tY}\right)_{22} \right]
\end{align}
and calculate its first and second partial derivatives at $(0,0;\lambda)$
\begin{align}
\nonumber
& \tfrac{\partial \L}{\partial s}|_{(0,0;\lambda)} = C(c_{11}-c_{22}) \RE\left[(1 - \lambda\ii)(c_{12}+c_{21})\right],
\\
\nonumber
& \tfrac{\partial \L}{\partial t}|_{(0,0;\lambda)} = C(c_{11}-c_{22}) \RE\left[(1 - \lambda\ii)\ii(c_{12}-c_{21})\right],
\\
\nonumber
& \tfrac{\partial^2 \L}{\partial s^2}|_{(0,0;\lambda)} = 2C \RE\left[(1 - \lambda\ii)((c_{11}-c_{22})^2 - (c_{12}+c_{21})^2)\right],
\\
\label{prod10}
& \tfrac{\partial^2 \L}{\partial t^2}|_{(0,0;\lambda)} = 2C \RE\left[(1 - \lambda\ii)((c_{11}-c_{22})^2 + (c_{12}-c_{21})^2)\right],
\\
\nonumber
& \tfrac{\partial^2 \L}{\partial s \partial t}|_{(0,0;\lambda)} = -2C \RE\left[(1 - \lambda\ii)\ii(c_{12}-c_{21})(c_{12}+c_{21})\right],
\\
\nonumber
& \tfrac{\partial^2 \L}{\partial s \partial \lambda}|_{(0,0;\lambda)} = -2C(c_{11}-c_{22}) \RE\left[\ii(c_{12}+c_{21})\right],
\\
\nonumber
& \tfrac{\partial^2 \L}{\partial t \partial \lambda}|_{(0,0;\lambda)} = 2C(c_{11}-c_{22}) \RE\left[c_{12}-c_{21}\right].
\end{align}

Before invoking Lagrange's theorem, we need to investigate whether the constraint qualification holds, i.e., whether $\nabla g = (\tfrac{\partial^2 \L}{\partial s \partial \lambda}, \tfrac{\partial^2 \L}{\partial t \partial \lambda})$ does not vanish at $(0,0)$. Suppose on the contrary that is \emph{does} vanish, what requires that $\RE\left[\ii(c_{12} + c_{21}) \right] = 0 = \RE\left[c_{12} - c_{21} \right]$, which is in turn equivalent to $c_{21} = \overline{c}_{12}$, i.e., to $M$ being Hermitian. Notice, however, that in such a case the matrix $\ee^{sX + tY}M\ee^{-sX - tY}$ is Hermitian for \emph{any} $s,t$ and hence we would have $g \equiv 0$ (and not just $\nabla g(0,0) = 0$). But this means that $f$ would in fact attain an \emph{unconstrained} maximum at $(0,0)$, and so its first derivatives would vanish:
\begin{align*}
& 0 = \tfrac{\partial f}{\partial s}|_{(0,0)} = \tfrac{\partial \L}{\partial s}|_{(0,0;0)} =  C(c_{11}-c_{22}) \RE\left[c_{12}+c_{21}\right],
\\
& 0 = \tfrac{\partial f}{\partial t}|_{(0,0)} = \tfrac{\partial \L}{\partial t}|_{(0,0;0)} =  C(c_{11}-c_{22}) \RE\left[\ii(c_{12}-c_{21})\right].
\end{align*}
This, together with $c_{11} \neq c_{22}$ and $c_{21} = \overline{c}_{12}$, yields $c_{12} = c_{21} = 0$. However, in such a case the Hessian of $f$ at $(0,0)$ would be equal to $2C(c_{11}-c_{22})^2 I_2$ and therefore be positive-definite, in contradiction with $f$ attaining a \emph{maximum} at that point.

In this way we have shown that $\nabla g(0,0) \neq 0$. On the strength of Lagrange's theorem, there exists $\lambda_0$ such that $(0,0;\lambda_0)$ is a stationary point of $\L$. By the first two formulas of (\ref{prod10}) and the assumption that $c_{11} \neq c_{22}$, this means that $\RE\left[(1-\lambda_0\ii)(c_{12} + c_{21}) \right] = 0 = \RE\left[(1-\lambda_0\ii)\ii(c_{12} - c_{21}) \right]$, which is in turn equivalent to
\begin{align}
\label{prod11}
c_{21} = \ee^{-2\ii \varphi} \, \overline{c}_{12},
\end{align}
where we have introduced $\varphi \vc \Arg(1 - \lambda_0 \ii)$. For later use, however, it is convenient to rewrite (\ref{prod11}) as
\begin{align}
\label{prod12}
c_{12} = \rho \ee^{\ii(\eta - \varphi)} \quad \textnormal{and} \quad c_{21} = -\rho \ee^{-\ii(\eta + \varphi)}
\end{align}
for some angle $\eta$ and $\rho > 0$ (we cannot have $\rho = 0$, because we would arrive at the same contradiction as when investigating the case of vanishing $\nabla g$).

In order to arrive at the final contradiction, let us perform the second derivative test, which in this case requires finding the bordered Hessian of $\L$ at $(0,0;\lambda_0)$. On the strength of (\ref{prod10}) and (\ref{prod12}), after some tedious but straightforward calculations we obtain
\begin{align*}
\left[ \begin{smallmatrix} 0 & \tfrac{\partial^2 \L}{\partial s \partial \lambda} & \tfrac{\partial^2 \L}{\partial t \partial \lambda}
\\
\tfrac{\partial^2 \L}{\partial s \partial \lambda} & \tfrac{\partial^2 \L}{\partial s^2} & \tfrac{\partial^2 \L}{\partial s \partial t}
\\
\tfrac{\partial^2 \L}{\partial t \partial \lambda} & \tfrac{\partial^2 \L}{\partial s \partial t} & \tfrac{\partial^2 \L}{\partial t^2}
\end{smallmatrix} \right]_{(0,0;\lambda_0)} = \tfrac{2C\cos\varphi}{\sqrt{1+\lambda_0^2}} \left[ \begin{smallmatrix} 0 & 2\rho\sqrt{1+\lambda_0^2}(c_{11} - c_{22})\sin\eta & 2\rho\sqrt{1+\lambda_0^2}(c_{11} - c_{22})\cos\eta
\\
2\rho\sqrt{1+\lambda_0^2}(c_{11} - c_{22})\sin\eta & (c_{11} - c_{22})^2+4\rho^2\sin^2\eta & 4\rho^2\sin\eta\cos\eta
\\
2\rho\sqrt{1+\lambda_0^2}(c_{11} - c_{22})\cos\eta & 4\rho^2\sin\eta\cos\eta & (c_{11} - c_{22})^2+4\rho^2\cos^2\eta
\end{smallmatrix} \right]
\end{align*}
Since we are dealing with a bivariate function subject to a single constraint, we are interested only in the (sign of the) bordered Hessian determinant, which turns out to be
\begin{align*}
& 32 C^3\rho^2(c_{11} - c_{22})^2 \sqrt{1+\lambda_0^2}\cos^3\varphi \left| \begin{smallmatrix} 0 & \sin\eta & \cos\eta
\\
\sin\eta & (c_{11} - c_{22})^2+4\rho^2\sin^2\eta & 4\rho^2\sin\eta\cos\eta
\\
\cos\eta & 4\rho^2\sin\eta\cos\eta & (c_{11} - c_{22})^2+4\rho^2\cos^2\eta
\end{smallmatrix} \right|
\\
& = -32 C^3\rho^2(c_{11} - c_{22})^4 \sqrt{1+\lambda_0^2}\cos^3\varphi.
\end{align*}
This expression has the same sign as $-\sqrt{1+\lambda^2_0}\cos\varphi$, which by the definition of $\varphi$ is equal to $-1$. But according to the second derivative test in the considered case (bivariate function, single constraint), this would imply that $f$ attains a constrained \emph{minimum} at $(0,0)$, what is a contradiction.

This finishes the proof that $c_{11}$ and $c_{22}$ --- and in fact all $c_{jj}$'s --- are equal.

\subsubsection*{Step 4: Proving that $c_{jk} = \ee^{-2\ii\varphi}\bar{c}_{kj}$.}

We can still extract some more information about the matrix $U_1 = [c_{jk}]$ introduced in (\ref{prod4}). To this end, we consider yet another subproblem of (\ref{max1}), this time involving functions of $n(n-1)$ variables $s_{jk}$, $t_{jk}$ indexed by $1 \leq j < k \leq n$, namely:
\begin{align}
\label{max3}
\textnormal{maximize} \quad & \tilde{f}((s_{jk}),(t_{jk})) \vc \RE\left[\ee^{-\ii \theta} \Pid\left(\exp\left(\sum\nolimits_{j<k} s_{jk}X_{jk} + t_{jk}Y_{jk}\right) U_0\right) \right]
\\
\nonumber
\textnormal{subject to} \quad & \tilde{g}((s_{jk}),(t_{jk}))) \vc \IM\left[\ee^{-\ii \theta} \Pid\left(\exp\left(\sum\nolimits_{j<k} s_{jk}X_{jk} + t_{jk}Y_{jk}\right) U_0\right) \right] = 0,
\end{align}
where $X_{jk}$ (resp. $Y_{jk}$) is the $n$-by-$n$ skew-Hermitian matrix whose only non-zero entries are the $(j,k)$-th entry equal to $-1$ (resp. $\ii$) and the $(k,j)$-th entry equal to $1$ (resp. $\ii$)\footnote{For example, for $n=3$ we have $$X_{12} = \left[\begin{smallmatrix} 0 & -1 & 0 \\ 1 & 0 & 0 \\ 0 & 0 & 0 \end{smallmatrix} \right], \quad X_{13} = \left[\begin{smallmatrix} 0 & 0 & -1 \\ 0 & 0 & 0 \\ 1 & 0 & 0 \end{smallmatrix} \right], \quad X_{23} = \left[\begin{smallmatrix} 0 & 0 & 0 \\ 0 & 0 & -1 \\ 0 & 1 & 0 \end{smallmatrix} \right], \quad Y_{12} = \left[\begin{smallmatrix} 0 & \ii & 0 \\ \ii & 0 & 0 \\ 0 & 0 & 0 \end{smallmatrix} \right], \quad Y_{13} = \left[\begin{smallmatrix} 0 & 0 & \ii \\ 0 & 0 & 0 \\ \ii & 0 & 0 \end{smallmatrix} \right], \quad Y_{23} = \left[\begin{smallmatrix} 0 & 0 & 0 \\ 0 & 0 & \ii \\ 0 & \ii & 0 \end{smallmatrix} \right].$$}.

Similarly as in the previous step, by the very definition of $U_0$, one solution to problem (\ref{max3}) is attained at the point $O$ with all $s$-coordinates and $t$-coordinates equal to $0$. Just like before, let us construct the Lagrangian function
\begin{align}
\label{prod13}
\tilde{\L}((s_{jk}),(t_{jk});\lambda) & \vc \tilde{f}((s_{jk}),(t_{jk})) + \lambda \tilde{g}((s_{jk}),(t_{jk}))
\\
\nonumber
& = \RE\left[(1 - \lambda \ii) \Pid\left(\exp\left(\sum\nolimits_{j<k} s_{jk}X_{jk} + t_{jk}Y_{jk}\right) U_1\right) \right],
\end{align}
where we have used (\ref{prod4}). Since the diagonal entries of $U_1$ are equal, $c_{11} = \ldots = c_{nn} \cv c > 0$, it is not hard to obtain that
\begin{align}
\nonumber
& \tfrac{\partial \tilde{\L}}{\partial s_{jk}}|_{(O;\lambda)} = c^{n-1}\RE\left[(1 - \lambda\ii)\tr(X_{jk}U_1)\right] = c^{n-1}\RE\left[(1 - \lambda\ii)(c_{jk} - c_{kj})\right],
\\
\label{prod14}
& \tfrac{\partial \tilde{\L}}{\partial t_{jk}}|_{(O;\lambda)} = c^{n-1}\RE\left[(1 - \lambda\ii)\tr(Y_{jk}U_1)\right] = c^{n-1}\RE\left[(1 - \lambda\ii)\ii(c_{jk} + c_{kj})\right],
\\
\nonumber
& \tfrac{\partial \tilde{g}}{\partial s_{jk}}|_O = \tfrac{\partial^2 \tilde{\L}}{\partial s_{jk} \partial \lambda}|_{(O;\lambda)} = - c^{n-1}\RE\left[ \ii \tr(X_{jk}U_1)\right] = -c^{n-1}\RE\left[\ii(c_{jk} - c_{kj})\right],
\\
\nonumber
& \tfrac{\partial \tilde{g}}{\partial t_{jk}}|_O = \tfrac{\partial^2 \tilde{\L}}{\partial t_{jk} \partial \lambda}|_{(O,\lambda)} = - c^{n-1}\RE\left[ \ii \tr(Y_{jk}U_1)\right] = c^{n-1}\RE\left[c_{jk} + c_{kj}\right].
\end{align}

Before invoking Lagrange's theorem, let us separately investigate the case when $\nabla \tilde{g}$ vanishes at $O$, i.e., when $\RE\left[\ii(c_{jk} - c_{kj})\right] = 0 = \RE\left[c_{jk} + c_{kj}\right]$ for every $j < k$, what is in turn equivalent to
\begin{align}
\label{prod15}
\forall j \neq k \quad c_{jk} = -\bar{c}_{kj}.
\end{align}
We claim that in such a case $\tilde{g}$ vanishes on every $s$-coordinate axis and on every $t$-coordinate axis, i.e., we claim that for any fixed $j < k$
\begin{align}
\label{prod16}
& \IM\left[\Pid\left(\ee^{s_{jk}X_{jk}} U_1\right) \right] = 0  \quad \textnormal{for all } s_{jk}
\\
\label{prod17}
& \IM\left[\Pid\left(\ee^{t_{jk}Y_{jk}} U_1\right) \right] = 0  \quad \textnormal{for all } t_{jk}.
\end{align}

Indeed, let $M$ denote the principal $2$-by-$2$ submatrix of $U_1$ associated with the $j$-th and $k$-th rows and columns, which by (\ref{prod15}) and the previous step of the proof can be expressed as follows
\begin{align*}
M \vc \left[ \begin{array}{cc} c_{jj} & c_{jk} \\ c_{kj} & c_{kk} \end{array} \right] = \left[ \begin{array}{cc} c & -\bar{d} \\ d & c \end{array} \right],
\end{align*}
where $c > 0$ and $d$ is some complex number. We can now rewrite the left-hand sides of (\ref{prod16},\ref{prod17}) as
\begin{align*}
& \IM\left[\Pid\left(\ee^{sX_{jk}} U_1\right) \right] = c^{n-2}\IM\left[\left(\ee^{sX} M\right)_{11}\left(\ee^{sX} M\right)_{22} \right],
\\
& \IM\left[\Pid\left(\ee^{tY_{jk}} U_1\right) \right] = c^{n-2}\IM\left[\left(\ee^{tY} M\right)_{11}\left(\ee^{tY} M\right)_{22} \right],
\end{align*}
where $X \vc \left[\begin{smallmatrix} 0 & -1 \\ 1 & 0 \end{smallmatrix} \right]$ and $Y \vc \left[\begin{smallmatrix} 0 & \ii \\ \ii & 0 \end{smallmatrix} \right]$ and we have renamed $s_{jk}$, $t_{jk}$ as $s$, $t$ for simplicity. The $2$-by-$2$ matrices involved can be easily calculated
\begin{align*}
& \ee^{sX} M = \left[ \begin{array}{cc} \cos s & -\sin s \\ \sin s & \cos s \end{array} \right] \left[ \begin{array}{cc} c & -\bar{d} \\ d & c \end{array} \right] = \left[ \begin{array}{cc} c \cos s - d \sin s  & \ast \\ \ast & c \cos s - \bar{d} \sin s \end{array} \right],
\\
& \ee^{tY} M = \left[ \begin{array}{cc} \cos t & \ii \sin t \\ \ii \sin t & \cos t \end{array} \right] \left[ \begin{array}{cc} c & -\bar{d} \\ d & c \end{array} \right] = \left[ \begin{array}{cc} c \cos t + \ii d \sin t  & \ast \\ \ast & c \cos t - \ii \bar{d} \sin t \end{array} \right],
\end{align*}
and we see that the imaginary parts of their diagonal products vanish, as claimed.

The vanishing of $\tilde{g}$ on the coordinate axes implies, in turn, that $\tilde{f}$ restricted to any coordinate axis attains a maximum at $0$. Hence its partial derivatives must all vanish at $O$, and with the help of (\ref{prod14}) we can write that
\begin{align*}
& 0 = \tfrac{\partial \tilde{f}}{\partial s_{jk}}|_{O} = \tfrac{\partial \tilde{\L}}{\partial s_{jk}}|_{(O;0)} = c^{n-1}\RE\left[c_{jk} - c_{kj}\right],
\\
& 0 = \tfrac{\partial \tilde{f}}{\partial t_{jk}}|_{O} = \tfrac{\partial \tilde{\L}}{\partial t_{jk}}|_{(O;0)} = c^{n-1}\RE\left[\ii(c_{jk} + c_{kj})\right]
\end{align*}
for every $j < k$. But this means that $c_{jk} = \bar{c}_{jk}$, which together with (\ref{prod15}) yields $c_{jk} = 0$ for all $j \neq k$. All in all, the assumption of $\nabla \tilde{g}$ vanishing at $O$ leads to $U_1 = cI$, which by the positivity of $c$ and the unitary of $U_1$ leaves $c = 1$ as the only option.

Let us now assume that $\nabla \tilde{g}$ does \emph{not} vanish at $O$. By Lagrange's theorem, there exists $\lambda_0$ such that $(O;\lambda_0)$ is a stationary point of $\tilde{\L}$. By the first two formulas of (\ref{prod14}), this means that $\RE\left[(1-\lambda_0\ii)(c_{jk} - c_{kj}) \right] = 0 = \RE\left[(1-\lambda_0\ii)\ii(c_{jk} + c_{kj}) \right]$ for every $j < k$, which is in turn equivalent to
\begin{align}
\label{prod18}
\forall j \neq k \quad c_{kj} = \ee^{-2\ii \varphi} \, \overline{c}_{jk},
\end{align}
where we have introduced $\varphi \vc \Arg(1 - \lambda_0 \ii)$. Notice that the case $U_1 = I$ is also included in (\ref{prod18}).

\subsubsection*{Step 5: Further narrowing down the form of $U_0$.}

So far we have proven that $U_0$ solving (\ref{max1}) can be written in the matrix form as (cf. (\ref{prod4}) and the previous two steps)
\begin{align}
\label{prod19}
U_0 & = \ee^{\ii \tfrac{\theta}{n}} \left[ \begin{array}{ccccc} c & c_{12} & c_{13} & \cdots & c_{1n} \\ \overline{c}_{12} \ee^{-2\ii\varphi} & c & c_{23} & \cdots & c_{2n} \\ \overline{c}_{13} \ee^{-2\ii\varphi} & \overline{c}_{23} \ee^{-2\ii\varphi} & c & \cdots & c_{3n} \\ \vdots & \vdots & \vdots & \ddots & \vdots  \\ \overline{c}_{1n} \ee^{-2\ii\varphi} & \overline{c}_{2n} \ee^{-2\ii\varphi} & \overline{c}_{3n} \ee^{-2\ii\varphi} & \cdots & c \end{array} \right]
\\
\nonumber
& = \ee^{\ii \left(\tfrac{\theta}{n} - \varphi\right)} \left[ \begin{array}{ccccc} c \ee^{\ii\varphi} & c_{12} \ee^{\ii\varphi} & c_{13} \ee^{\ii\varphi} & \cdots & c_{1n} \ee^{\ii\varphi} \\ \overline{c}_{12} \ee^{-\ii\varphi} & c \ee^{\ii\varphi} & c_{23} \ee^{\ii\varphi} & \cdots & c_{2n} \ee^{\ii\varphi} \\ \overline{c}_{13} \ee^{-\ii\varphi} & \overline{c}_{23} \ee^{-\ii\varphi} & c \ee^{\ii\varphi} & \cdots & c_{3n} \ee^{\ii\varphi} \\ \vdots & \vdots & \vdots & \ddots & \vdots  \\ \overline{c}_{1n} \ee^{-\ii\varphi} & \overline{c}_{2n} \ee^{-\ii\varphi} & \overline{c}_{3n} \ee^{-\ii\varphi} & \cdots & c \ee^{\ii\varphi} \end{array} \right] = \ee^{\ii \left(\tfrac{\theta}{n} - \varphi\right)} \left(H + \ii c \sin\varphi I \right),
\end{align}
where $H$ is a Hermitian matrix with all diagonal entries equal to $c \cos\varphi$.

The Hermiticity of $H$ together with the unitarity of $H + \ii c \sin\varphi I$ imply that $U_0$ has at most two eigenvalues
\begin{align}
\label{prod20}
\lambda_\pm \vc \ee^{\ii \left(\tfrac{\theta}{n} - \varphi\right)} \left(\pm \sqrt{1 - c^2 \sin^2\varphi} + \ii c \sin\varphi \right),
\end{align}
which means that $U_0$ must be a linear combination of at most two orthogonal projections
\begin{align}
\label{prod21}
U_0 = \lambda_+(I - VV^\dag) + \lambda_-VV^\dag,
\end{align}
where $V$ is some $n$-by-$m$ matrix whose columns are orthonormal, with $m$ being the multiplicity of $\lambda_-$ in the spectrum of $U_0$. 

If $m=0$ or $m=n$, $U_0$ becomes a scalar matrix, but since $G(U_0) = 0$ and $\det U_0 = 1$ this requires $\theta = 0$ and $U_0 = I$ or, provided $n$ is even, $U_0 = -I$. Notice that this trivial case agrees with the statement of Theorem \ref{main} (for $\alpha = 0$, to be precise).

Assuming from now on that $1 < m < n$, let us rewrite (\ref{prod21}) with the help of (\ref{prod20}) as
\begin{align}
\label{prod22}
\ee^{\ii \left(\varphi - \tfrac{\theta}{n}\right)} U_0 = \left(\sqrt{1 - c^2 \sin^2\varphi} + \ii c \sin\varphi \right) I -2\sqrt{1 - c^2 \sin^2\varphi}  \, VV^\dag.
\end{align}
From (\ref{prod19}) we obtain that, for every $j$-th diagonal entry of $H$,
\begin{align}
\label{prod23}
c \cos\varphi = \sqrt{1 - c^2 \sin^2\varphi} \left( 1 - 2 \sum\nolimits_{k=1}^m |v_{jk}|^2 \right).
\end{align}
This means that, firstly, the sum $\sum_{k=1}^m |v_{jk}|^2$ does not depend on $j$ and hence
\begin{align}
\label{prod24}
\sum\nolimits_{k=1}^m |v_{jk}|^2 = \frac{m}{n}.
\end{align}
Secondly, we can now solve (\ref{prod23}) for $c$, obtaining
\begin{align}
\label{prod25}
c = \frac{|n - 2m|}{\sqrt{n^2 \cos^2 \varphi + (n - 2m)^2\sin^2\varphi}}.
\end{align}

Let us now reexpress the two eigenvalues (\ref{prod20}) substituting (\ref{prod25}) for $c$ (and bearing in mind that $\cos \varphi = 1/\sqrt{1 + \lambda_0^2} > 0$, cf. the sentence following formula (\ref{prod18}))
\begin{align}
\label{prod26}
\lambda_\pm = \ee^{\ii \left(\tfrac{\theta}{n} - \varphi\right)} \frac{\pm n \cos\varphi + \ii |n - 2m|\sin\varphi}{\sqrt{n^2 \cos^2 \varphi + (n - 2m)^2\sin^2\varphi}}.
\end{align}
It is convenient to introduce an auxilliary complex number of modulus one 
\begin{align*}
\xi \vc \frac{n \cos\varphi + \ii |n - 2m|\sin\varphi}{\sqrt{n^2 \cos^2 \varphi + (n - 2m)^2\sin^2\varphi}}
\end{align*}
and write the eigenvalues as
\begin{align*}
\lambda_+ = \ee^{\ii \left(\tfrac{\theta}{n} - \varphi\right)} \xi \qquad \textnormal{and} \qquad \lambda_- = - \ee^{\ii \left(\tfrac{\theta}{n} - \varphi\right)} \bar{\xi}.
\end{align*}
In this language, formula (\ref{prod22}) becomes
\begin{align}
\label{prod27}
\ee^{\ii \left(\varphi - \tfrac{\theta}{n}\right)} U_0 = \xi(I - VV^\dag) - \bar{\xi}VV^\dag = \xi \left( I - \left( 1 + \bar{\xi}^2\right) VV^\dag \right).
\end{align}

We still have not used the fact that $\det U_0 = 1$. Since the determinant is equal to the product of eigenvalues (with multiplicities), we have that
\begin{align*}
\ee^{\ii \left(n\varphi - \theta\right)} = \xi^{n-m} (-\bar{\xi})^m = \xi^n (-\bar{\xi}^2)^m,
\end{align*}
and so, taking the $n$-th root,
\begin{align*}
\ee^{\ii \left(\varphi - \tfrac{\theta}{n}\right)} = \ee^{\ii \tfrac{2l\pi}{n}} \xi (-\bar{\xi}^2)^{\tfrac{m}{n}}
\end{align*}
for some $l \in \sZ$. But this means that (\ref{prod27}) can be simplified to
\begin{align}
\label{prod28}
\ee^{\ii \tfrac{2l\pi}{n}} (-\bar{\xi}^2)^{\tfrac{m}{n}} U_0 = I - ( 1 + \bar{\xi}^2) VV^\dag
\end{align}
To further simplify, introduce a new angle $\alpha \in (-\pi, \pi]$ such that $\ee^{\ii \alpha} \vc -\xi^2$ and transform (\ref{prod28}) into
\begin{align}
\label{prod29}
U_0 = \ee^{\ii \tfrac{m \alpha - 2l\pi}{n}} \left(I - \left( 1 - \ee^{-\ii \alpha} \right) VV^\dag \right)
\end{align}

Recall that in Step 2 we have assumed without loss of generality that the diagonal entries of $U_0$ have all the same argument, what was based on observation (\ref{obs}). Using the latter again to \emph{drop} the former assumption, we find that any $U_0$ solving constrained maximization problem (\ref{max1}) must be of the form
\begin{align}
\label{prod30}
U_0 = \left(I - \left( 1 - \ee^{-\ii \alpha} \right) VV^\dag \right) \diag(\ee^{\ii \alpha_1},\ldots,\ee^{\ii \alpha_n}), \quad \textnormal{where} \quad \ee^{\ii \sum_k \alpha_k } = \ee^{\ii m \alpha}
\end{align}
Evaluating the diagonal product we obtain, using (\ref{prod24}),
\begin{align}
\label{prod31}
\Pid(U_0) =  \ee^{\ii m \alpha} \left( 1 - \left( 1 - \ee^{-\ii \alpha} \right) \tfrac{m}{n} \right)^n.
\end{align}

\subsubsection*{Step 6: The reduced constrained maximization problem.}

At this point we have managed to reduce the original constrained maximization problem (\ref{max1}) on $\textrm{SU}(n)$ to the following one
\begin{align}
\label{max4}
\textnormal{maximize} \quad & \RE(\ee^{-\ii \theta} \Gamma(\alpha,m)) \quad \textnormal{over } \alpha \in [-\pi,\pi], \ m \in \{1,\ldots,n-1\}
\\
\nonumber
\textnormal{subject to} \quad & \IM(\ee^{-\ii \theta} \Gamma(\alpha,m)) = 0,
\end{align}
where the bivariate function involved is $\Gamma(\alpha,m) \vc \ee^{\ii m \alpha} \left( 1 - \left( 1 - \ee^{-\ii \alpha} \right) \tfrac{m}{n} \right)^n$. Let us then \emph{forget} about the original problem (for a while) and focus on (\ref{max4}), or rather on its slightly more general and convenient, continuous-second-variable version:
\begin{align}
\label{max5}
\textnormal{maximize} \quad & \RE(\ee^{-\ii \theta} \Gamma(\alpha,y)) \quad \textnormal{over } \alpha \in [-\pi,\pi], \ y \in [1,n-1]
\\
\nonumber
\textnormal{subject to} \quad & \IM(\ee^{-\ii \theta} \Gamma(\alpha,y)) = 0,
\end{align}
where from now on $\Gamma: [-\pi,\pi] \times [1,n-1] \rightarrow \sC$ is defined as
\begin{align}
\label{Gammadef}
\Gamma(\alpha,y) \vc \ee^{\ii y \alpha} \left( 1 - \left( 1 - \ee^{-\ii \alpha} \right) \tfrac{y}{n} \right)^n.
\end{align}

In other words, we are interested in how far (from the origin) stretches the $\theta$-rotated positive real axis intersected with the image $\Gamma([-\pi,\pi] \times [1,n-1])$. Notice that $\Gamma(\alpha(\theta),1) = \ee^{\ii\theta}|\gamma(\alpha(\theta))|$, with $\alpha(\theta)$ defined in Remark \ref{rem2} (cf. the second paragraph of Step 2 above), so the constrained maximum solving (\ref{max5}) is positive.

Observe first that 
\begin{align}
\label{prod32}
\Gamma(\alpha,n-y) = \Gamma(-\alpha,y) = \overline{\Gamma(\alpha,y)},
\end{align}
so without loss of generality we can restrict $\Gamma$ to $\Delta \vc [0,\pi] \times [1,n-1]$. We claim that the search space in problem (\ref{max5}) can be narrowed down to $\Gamma(\partial \Delta)$.

In order to prove the claim, let us first show that the Jacobian of $\Gamma$ is positive on the interior $\Delta^\circ = (0,\pi) \times (1,n-1)$ (where we tacitly identify $\sC$ with $\sR^2$). Indeed, observe that
\begin{align*}
J[\Gamma] \vc \left| \begin{array}{cc} \partial_\alpha \RE \Gamma & \partial_\alpha \IM \Gamma \\ \partial_y \RE \Gamma & \partial_y \IM \Gamma \end{array} \right| = \RE(\partial_\alpha \Gamma) \IM(\partial_y \Gamma) - \IM(\partial_\alpha \Gamma) \RE(\partial_y \Gamma) = \IM ( \overline{\partial_\alpha \Gamma} \, \partial_y \Gamma).
\end{align*}
Calculating the first derivatives, we obtain
\begin{align*}
& \partial_\alpha \Gamma = \ii y \left(1-\tfrac{y}{n}\right) \left( 1 - \ee^{-\ii \alpha} \right) \ee^{\ii y \alpha} \left( 1 - \left( 1 - \ee^{-\ii \alpha} \right) \tfrac{y}{n} \right)^{n-1},
\\
& \partial_y \Gamma = \left[ \ii \alpha - \left( 1 + \tfrac{\ii y \alpha}{n} \right) \left( 1 - \ee^{-\ii \alpha} \right) \right] \ee^{\ii y \alpha} \left( 1 - \left( 1 - \ee^{-\ii \alpha} \right) \tfrac{y}{n} \right)^{n-1},
\end{align*}
and hence
\begin{align*}
J[\Gamma] & = \IM ( \overline{\partial_\alpha \Gamma} \, \partial_y \Gamma) = |\Gamma|^{2n-2} y \left(1-\tfrac{y}{n}\right) \IM\left[ -\ii \left(1 - \ee^{\ii \alpha}\right) \left( \ii \alpha - \left( 1 + \tfrac{\ii y \alpha}{n} \right) \left(1 - \ee^{-\ii \alpha} \right) \right) \right]
\\
& = |\Gamma|^{2n-2} y \left(1-\tfrac{y}{n}\right) \IM\left[ \alpha \left(1 - \ee^{\ii \alpha}\right) + \left( \ii - \tfrac{y \alpha}{n} \right) \left|1 - \ee^{-\ii \alpha} \right|^2  \right]
\\
& = |\Gamma|^{2n-2} y \left(1-\tfrac{y}{n}\right) \left[ -\alpha \sin \alpha + \left|1 - \ee^{-\ii \alpha} \right|^2 \right]
\\
& = |\Gamma|^{2n-2} y \left(1-\tfrac{y}{n}\right) \left( 2 - 2\cos\alpha - \alpha \sin \alpha \right)
\\
& = |\Gamma|^{2n-2} y \left(1-\tfrac{y}{n}\right) \sin \alpha \left( 2\tan \tfrac{\alpha}{2} - \alpha \right) > 0,
\end{align*}
because $|\Gamma|$ vanishes only at the point $(\pi,\tfrac{n}{2}) \not\in \Delta^\circ$.

The fact that $J[\Gamma] > 0$ on $\Delta^\circ$ implies that $\Gamma(\Delta^\circ)$ is open, and hence contained in the interior $\Gamma(\Delta)^\circ$. But a solution to (\ref{max5}) cannot be attained at an interior point, because around such a point there exists a neighborhood contained in $\Gamma(\Delta)$, within which we could always move farther from the origin along the $\theta$-rotated positive real axis. This in turn means that we can narrow the search space in (\ref{max5}) \emph{at least}\footnote{This reasoning actually enables us to narrow the search space in (\ref{max5}) even to $\partial \Gamma(\Delta)$, but its superset $\Gamma(\partial \Delta)$ is sufficient for our purposes and easier to analyze.} to $\Gamma(\partial \Delta)$, because
\begin{align}
\label{gt}
\Gamma(\Delta) \setminus \Gamma(\Delta^\circ) \subset \Gamma(\Delta \setminus \Delta^\circ) = \Gamma(\partial \Delta).
\end{align}

\begin{figure}[h]
\centering
  \includegraphics[width=1.05\linewidth]{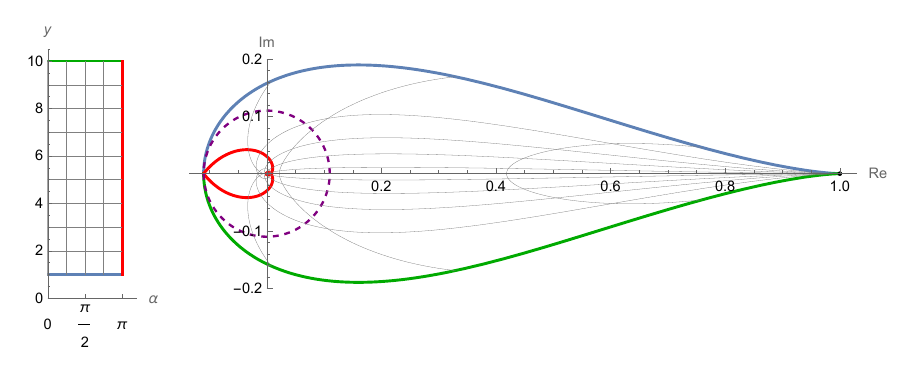}
\caption{\label{fig3}}The right panel shows the image of $\Delta \vc [0,\pi] \times [1,n-1]$ (shown in the left panel) under the map $\Gamma$ for $n=12$. Additionally shown is the circle of radius $(1-2/n)^n$ centered at $0$.
\end{figure}

Let us then investigate what happens with the boundary $\partial \Delta$ under the map $\Gamma$ (see Figure \ref{fig3} for illustration). The four sides of $\partial \Delta = \itSigma_\textup{L} \cup \itSigma_\textup{R} \cup \itSigma_\textup{B} \cup \itSigma_\textup{T}$ become curves parametrized as follows: 
\begin{itemize}
\item The left side $\itSigma_\textup{L}$ becomes squished to a single point, $\Gamma(0,y) = 1$ for all $y$.
\item The right side becomes: $\Gamma(\itSigma_\textup{R}) = \{ \Gamma(\pi,y) = \ee^{\ii y \pi} \left( 1 - \tfrac{2y}{n} \right)^n \, | \, y \in [1,n - 1]\}$.
\item The bottom side: $\Gamma(\itSigma_\textup{B}) = \{ \Gamma(\alpha,1) = \ee^{\ii \alpha} \left( 1 - \left( 1 - \ee^{-\ii \alpha} \right) \tfrac{1}{n} \right)^n \, | \, \alpha \in [0,\pi] \} = \gamma([0,\pi])$.
\item The top side: $\Gamma(\itSigma_\textup{T}) = \{ \Gamma(\alpha,n-1) = \Gamma(-\alpha,1) \, | \, \alpha \in [0,\pi]\} = \gamma([-\pi,0])$, where we have used (\ref{prod32}).
\end{itemize}

Observe that $\Gamma(\itSigma_\textup{L}) \cup \Gamma(\itSigma_\textup{B}) \cup \Gamma(\itSigma_\textup{T})$ is nothing but the image of the curve $\gamma$. Moreover, $\Gamma(\itSigma_\textup{R})$ is contained in the compact region enclosed by $\gamma$, because
\begin{align*}
|\Gamma(\pi,y)| = \left| 1 - \tfrac{2y}{n} \right|^n \leq (1 - \tfrac{2}{n})^n \leq |\gamma(\alpha)|
\end{align*}
for any $y \in [1,n-1]$ and $\alpha \in [-\pi,\pi]$, with equalities attained only at the intersection point of $\Gamma(\itSigma_\textup{R})$ with the image of $\gamma$, which is $-(1 - \tfrac{1}{n})^n = \gamma(\pm \pi) = \Gamma(\pm \pi, 1) = \Gamma(\pm \pi, n-1)$. Therefore, the set $\itSigma_\textup{R}$ can also be removed from the search space in (\ref{max5}).

This, in turn, demonstrates that the constrained maximum in problem (\ref{max5}), and hence in problems (\ref{max4}) and (\ref{max1}), is indeed equal to $|\gamma(\alpha(\theta))|$ with $\alpha(\theta)$ defined in Remark \ref{rem2}, for any chosen polar angle $\theta$.

\subsubsection*{Step 7: Final derivation of formula (\ref{main2}).}

The question remains what the value of the constrained maximum in (\ref{max1}) tells us about the form of the maximizer $U_0$, so far known to be given by (\ref{prod30}). Since its diagonal product is $\Pid(U_0) = \Gamma(\alpha,m)$, what we are asking now is: For which $(\alpha, m)$ the point $\Gamma(\alpha,m)$ lies on the image of $\gamma$?

We claim there are only three (families of) possibilities:
\begin{enumerate}[(i)]
\item $(\alpha, 1)$ for any $\alpha$, simply because $\Gamma(\alpha,1) = \gamma(\alpha)$,
\item $(\alpha, n-1)$ for any $\alpha$, because $\Gamma(\alpha,n-1) = \Gamma(-\alpha,1) = \gamma(-\alpha)$ by (\ref{prod32}),
\item $(0,m)$ for any $m$, because trivially $\Gamma(0,m) = 1 = \gamma(0)$.
\end{enumerate}

In order to prove that for any $\alpha \neq 0$ and any $m \in \{2,\ldots,n-2\}$ the point $\Gamma(\alpha, m)$ does \emph{not} lie on the image of $\gamma$, suppose on the contrary that it does, i.e., that $\Gamma(\alpha, m) = \gamma(\alpha_\ast)$ for some $\alpha_\ast$. There are three options, all leading to contradition, namely
\begin{itemize}
\item If $\alpha \in (0,\pi)$, then $\gamma(\alpha_\ast)$ would be an interior point of $\Gamma(\Delta)$, which is impossible because then $|\gamma(\alpha_\ast)|$ would not be a constrained maximum for problem (\ref{max5}) with $\theta = \theta(\alpha_\ast)$ given by (\ref{polar2}).
\item If $\alpha \in (-\pi,0)$, then by (\ref{prod32}) $\Gamma(\alpha, m) = \Gamma(-\alpha, n-m)$, and so $\gamma(\alpha_\ast)$ would again be an interior point of $\Gamma(\Delta)$, just like in the previous case.
\item If $\alpha = \pm\pi$, then we would have $(1-\tfrac{2}{n})^n \leq |\gamma(\alpha_\ast)| = |\Gamma(\pm\pi, m)| = |1-\tfrac{2m}{n}|^n < (1-\tfrac{2}{n})^n$, which is absurd.
\end{itemize}

Possibility (i) means we put $m=1$ in formula (\ref{prod30}), thus obtaining (\ref{main2}).

Possibility (ii) means we put $m=n-1$ in formula (\ref{prod30}), but it can be then transformed into (\ref{main2}) with $\tilde{\alpha} \vc -\alpha$ in place of $\alpha$. Indeed, since $V$ is an $n$-by-$(n-1)$ matrix whose columns are orthonormal, there exists a column vector $v \in \sC^n$ such that $VV^\dag = I - vv^\dag$. Its components must satisfy $|v_j| = 1/\sqrt{n}$ for all $j$ on the strength of (\ref{prod24}) and one can thus write that
\begin{align*}
U_0 & = \left(I - \left( 1 - \ee^{-\ii \alpha} \right) VV^\dag \right) \diag(\ee^{\ii \alpha_1},\ldots,\ee^{\ii \alpha_n}) 
\\
& = \left(I - \left( 1 - \ee^{-\ii \alpha} \right) (I - vv^\dag) \right) \diag(\ee^{\ii \alpha_1},\ldots,\ee^{\ii \alpha_n})
\\
& = \left(I - \left( 1 - \ee^{\ii \alpha} \right) vv^\dag \right)\diag(\ee^{\ii (\alpha_1-\alpha)},\ldots,\ee^{\ii( \alpha_n-\alpha)})
\\
& = \left(I - \left( 1 - \ee^{-\ii \tilde{\alpha}} \right) vv^\dag \right)\diag(\ee^{\ii \tilde{\alpha}_1},\ldots,\ee^{\ii \tilde{\alpha}_n}),
\end{align*}
where $\tilde{\alpha}_k \vc \alpha_k - \alpha$ satisfy $\ee^{\ii \sum\nolimits_k \tilde{\alpha}_k} = \ee^{\ii \sum\nolimits_k \alpha_k}\ee^{-\ii n\alpha} = \ee^{\ii (n-1)\alpha}\ee^{-\ii n\alpha} = \ee^{-\ii \alpha} = \ee^{\ii \tilde{\alpha}}$, as they should.

Finally, possibility (iii) means we put $\alpha = 0$ in formula (\ref{prod30}). This gives simply $U_0 = \diag(\ee^{\ii \alpha_1},\ldots,\ee^{\ii \alpha_n})$ with $\ee^{\ii \sum\nolimits_k \alpha_k} = 1$, erasing $m$ from the expression whatsoever and yielding a case included in formula (\ref{main2}) as well.

This concludes the proof of the second part of Theorem \ref{main}.

\subsubsection*{Step 8: Showing that $\Pid(\textup{SU}(n))$ is the entire compact region enclosed by $\gamma$.}

Let $R_\gamma$ denote the compact region enclosed by the curve $\gamma$. We have already shown that $\Pid(\textup{SU}(n)) \subset R_\gamma$, but we still need to show the converse inclusion. To this end, let us consider the following subclass of special unitary matrices
\begin{align}
\label{Amatrix}
A(\alpha,\omega) & \vc \left( I - (1 - \ee^{-\ii \alpha})v(\omega)v^\dag(\omega) \right)\diag(\ee^{\ii \alpha},1,\ldots,1), 
\\
\nonumber
\textnormal{with} \quad v(\omega) & \vc \big[ \begin{array}{cccc} \cos\omega & \tfrac{1}{\sqrt{n-1}}\sin\omega & \ldots & \tfrac{1}{\sqrt{n-1}}\sin\omega \end{array} \big]^T,
\end{align}
for any $\alpha \in S^1$ and any $\omega \in [0,\arctan\sqrt{n-1}]$. It is easy to check that any such matrix indeed belongs to $\textup{SU}(n)$, and so $\Pid(A(S^1 \times [0,\arctan\sqrt{n-1}])) \subset \Pid(\textup{SU}(n)) \subset R_\gamma$.

In order to show the converse inclusion $R_\gamma \subset \Pid(A(S^1 \times [0,\arctan\sqrt{n-1}]))$, suppose, on the contrary, that there exists $z_0 \in R_\gamma \setminus \Pid(A(S^1 \times [0,\arctan\sqrt{n-1}]))$, and so the map $\Pid \circ A: S^1 \times [0,\arctan\sqrt{n-1}] \rightarrow \sC \setminus \{z_0\}$ is well defined. Notice that
\begin{align*}
\Pid(A(\alpha,\omega)) = \ee^{\ii\alpha}\left[ 1 - (1 - \ee^{-\ii\alpha})\cos^2\omega\right]\left[ 1 - (1 - \ee^{-\ii\alpha})\tfrac{\sin^2\omega}{n-1}\right]^{n-1},
\end{align*}
and so in particular $\Pid(A(\alpha,\arctan\sqrt{n-1})) = \gamma(\alpha)$ for every $\alpha$. The point $z_0$ thus lies somewhere within the region enclosed by $\gamma$. Observe, however, that $\Pid(A(\alpha,0)) = 1$ for every $\alpha$, which means that $\Pid \circ A$, being continuous, is a \emph{homotopy} between $\gamma$ and the point 1, in contradiction with $\sC \setminus \{z_0\}$ not being simply connected. \qed

\medskip

The main result is thus proven. Let us finish with a corollary concerning the diagonal product of special \emph{orthogonal} matrices.
\begin{Cor}
\label{cor}
For any fixed $n \geq 1$ one has $\Pid(\textup{SO}(n)) = [-(1-2/n)^n,1]$. Moreover, for any $U \in \textup{SO}(n)$,
\begin{itemize}
\item $\Pid(U) = 1$ iff $U = \diag(\sigma_1,\ldots,\sigma_n)$, where $\forall j \ \sigma_j \in \{-1,1\}$ with $\prod\nolimits_j \sigma_j = 1$.
\item $\Pid(U) = -(1-2/n)^n$ iff $U = \left( I - 2 u u^T \right)\diag(\sigma_1,\ldots,\sigma_n)$, where $u \in \sR^n$ satisfies $|u_k| = \tfrac{1}{\sqrt{n}}$ for all $k$, and $\forall j \ \sigma_j \in \{-1,1\}$ with $\prod\nolimits_j \sigma_j = -1$.
\end{itemize}
\end{Cor}
\begin{proof}
To see that $\Pid(\textup{SO}(n)) \subset [-(1-2/n)^n,1]$, notice that by Theorem \ref{main}
\begin{align*}
\Pid(\textup{SO}(n)) \subset \Pid(\textup{SU}(n)) \cap \sR = [-(1-2/n)^n,1].
\end{align*}

To show the converse inclusion, consider the matrices $A(\pi,\omega)$ defined by (\ref{Amatrix}) for any $\omega \in [0,\arctan\sqrt{n-1}]$. It is easy to see that indeed $A(\pi,\omega) \in \textup{SO}(n)$ for any $\omega$. Furthermore, the continuous map $\Pid(A(\pi,\cdot)): [0,\arctan\sqrt{n-1}] \rightarrow \sR$ attains the value $1$ at $\omega = 0$ and the value $-(1-2/n)^n$ at $\omega = \arctan\sqrt{n-1}$, and so it attains all the intermediate values as well.

As for the form of $U \in \textup{SO}(n)$ extremizing the diagonal product, it follows directly from the second part of Theorem \ref{main}, in particular from formula (\ref{main2}) with $\alpha = 0$ and $\alpha = \pi$.
\end{proof}

\section*{Acknowledgements}
I am grateful to Shmuel Friedland and Karol {\.Z}yczkowski for encouraging me to pursue the above problem and for all their insightful comments. I also wish to thank Rafa{\l} Bistro\'n and Micha{\l} Eckstein for countless discussions. Last but not least, I want to express my gratitude to the anonymous user1551 on \url{math.stackexchange}, whose elaborate answer to the $\textup{SO}(n)$-version of the problem set me on track to find the presented proof.

\bibliographystyle{elsarticle-num}
\bibliography{diagonalproduct_bib}{}

\end{document}